\documentclass[a4paper, reqno, 11pt]{amsart}
\makeatletter

\@addtoreset{equation}{section}
\makeatother
\usepackage{setspace}
\usepackage{xcolor}
\usepackage{amsmath}
\usepackage{amssymb}
\usepackage{latexsym}
\usepackage{amsthm}
\usepackage{textcomp}
\usepackage{hyperref}
\newtheorem{Thm}{Theorem}[section]

\newtheorem{Lem}[Thm]{Lemma}
\newtheorem{Prop}[Thm]{Proposition}

\newtheorem{Cor}[Thm]{Corollary}

\theoremstyle{definition}

\newtheorem{Ass}[Thm]{Assumption}
\newtheorem{Def}[Thm]{Definition}
\newtheorem{Rem}[Thm]{Remark}
\newtheorem{Exa}[Thm]{Example}

\begin{document}

\title[Limit theorem for QAM]{Limit theorems for quasi-arithmetic means of  random variables with applications to point estimations for the Cauchy distribution}
\author[Y.Akaoka]{Yuichi Akaoka}
\thanks{A part of this paper consists of Y.~A.'s master's thesis \cite{Akaoka2020a}.}
\address{Department of Mathematics, Faculty of Science, Shinshu University}
\curraddr{Gunma bank}
\email{18ss101b@gmail.com}

\author[K. Okamura]{Kazuki Okamura}
\address{Department of Mathematics, Faculty of Science, Shizuoka University}
\email{okamura.kazuki@shizuoka.ac.jp}

\author[Y. Otobe]{Yoshiki Otobe}
\address{Department of Mathematics, Faculty of Science, Shinshu University}
\email{otobe@math.shinshu-u.ac.jp}

\date{\today}
\dedicatory{}

\maketitle

\begin{abstract}
We establish some limit theorems  for quasi-arithmetic means of  random variables. 
This class of means contains the arithmetic,  geometric and harmonic means. 
Our feature is that the generators of quasi-arithmetic means are allowed to be complex-valued, which makes  considerations for quasi-arithmetic means of random variables which could take negative values possible. 
Our motivation for the limit theorems  is finding simple estimators of the parameters of the Cauchy distribution. 
By applying the  limit theorems, we obtain some closed-form unbiased strongly-consistent estimators for the joint of the location and scale parameters of the Cauchy distribution, which are easy to compute and analyze.   
\end{abstract} 

\setcounter{tocdepth}{1}
\tableofcontents

\section{Introduction}

The strong law of large numbers is one of the most fundamental result in probability theory.  
It states that the arithmetic mean of  independent and identically distributed (i.i.d.) integrable real-valued random variables converges to the expecctation of the random variable almost surely. 
It is also important for point estimations for parameters of statistical models. 
Specifically, it could be a strongly consistent estimator of a parameter. 

If the integrable assumption fails, then the strong law of large numbers could fail. 
For example, the arithmetic mean of  i.i.d. random variables with the stable distribution does not converge to any constant almost surely. 
In particular, the distribution of the arithmetic mean of random variables with the Cauchy distribution is identical with the distribution of the random variable. 
Therefore we cannot use the arithmetic mean as a point estimator of the location parameter of the Cauchy distribution. 

The notion of means has been extended to more general classes.   
Notable notions of means other than the arithmetic means  are the geometric and harmonic means. 
By the arithmetic-geometric and geometric-harmonic mean inequalities, 
the arithmetic mean is larger than or equal to the geometric mean, and the geometric mean is larger than or equal to the harmonic mean. 
Therefore we might expect that the law of large numbers for the geometric  or harmonic mean holds even if the integrability fails.  
Pakes \cite{Pakes1999} considered asymptotic behaviors of the variances for the geometric and harmonic means of  non-negative random variables. 
To the best of our knowledge, the geometric mean has been considered for positive real numbers. 
However, in many statistical models, samples can take strictly negative real numbers.  

Kolmogorov \cite{Kolmogorov1930} proposed  axioms of means, 
and showed that if the  axioms of means hold for an $n$-ary operation on a set, 
then it has the form of a quasi-arithmetic mean appearing \eqref{f-mean} below.  
This class of quasi-arithmetic means contains the arithmetic, geometric and harmonic means. 
The axiomatic treatment of means and properties of the quasi-arithmetic means have been considered by many authors (e.g. Nagumo \cite{Nagumo1930}, de Finetti \cite{Finetti1931}, Acz\'el \cite{Aczel1948}). 
From a probabilistic viewpoint, 
Carvalho \cite{Carvalho2016} showed the central limit theorem for quasi-arithmetic means by a direct application of the delta method.  
Recently, Barczy and Burai \cite{Barczy2021} obtained strong law of large numbers and central limit theorems
for Bajraktarevic means of i.i.d. random variables using the Delta method. 
The class of Bajraktarevic means is strictly larger than the class of quasi arithmetic means. 
However they deal with real-valued  quasi-arithmetic means, 
and to the best of our knowledge,  {\it complex-valued} quasi-arithmetic means of random variables have not yet been considered.   

The main purposes of this paper are to establish limit theorems for some quasi-arithmetic means of random variables, and then apply them to point estimations for the location and scale parameters of the Cauchy distribution.   
We obtain that the {\it geometric} mean of i.i.d. random variables of the Cauchy distribution with location $\mu \in \mathbb{R}$ and scale $\sigma > 0$ converges to $\mu + \sigma i$ almost surely and in $L^p$ for every $0 < p < +\infty$, where $i$ denotes the imaginary unit.   
This is contrast to the {\it arithmetic} mean of them, which does {\it not} converges to any constant and is identical with the original random variable in distribution.   

By allowing the generator of quasi-arithmetic means to take {\it complex numbers}, 
we could define the geometric mean of  every  non-zero {\it real} numbers by allowing the powers of negative numbers to be {\it complex-valued}. 
Thus we can consider  the geometric mean of {\it real-valued} continuous random variables.    
The strong law of large numbers and the central limit theorem, which include \cite[Theorem 1]{Carvalho2016}  as a special case, are shown by the delta method. 
We consider asymptotic behaviors of the variances for the geometric mean, and 
extend \cite[Theorem 4]{Pakes1999} to our complex-valued setting.  
We also deal with a modification of the harmonic mean, more specifically, a class of quasi-arithmetic means generated by M\"obius transformations. 

We apply these limit theorems to unbiased, closed-form point estimations for the location and scale parameters of the Cauchy distribution. 
We adopt McCullagh's parametrization \cite{McCullagh1992, McCullagh1996}, which regards the location and scale parameters as a single complex-valued parameter. 
By this parametrization, we can estimate the location  and scale  simultaneously in the case that neither of the parameters is known.   
The main advantages of using the quasi-arithmetic means are that the estimators are {\it closed-form} and the various techniques for the arithmetic mean of random variables are easily applicable. 
For calculating the estimated value from samples, 
we do not need the permutations used in the order statistic, or, the numerical computations such as the Newton-Raphson method used in the maximal likelihood estimation.

\subsection{Framework}

For $z \in \mathbb{C}$, $\textup{Re}(z)$ and $\textup{Im}(z)$ are the real and imaginary parts of $z$ respectively. 
$|z|$ denotes the absolute value of $z$, that is, $|z| = \sqrt{\textup{Re}(z)^2+\textup{Im}(z)^2}$. 
Let $\mathbb H := \{x+yi: y > 0\}$. 
Let the closures of the upper-half plane $\mathbb H$ and the lower-half plane $-\mathbb H$ be $\overline{\mathbb H} := \{x+yi: y \ge 0\}$ and $\overline{-\mathbb H} := \{x+yi : y \le 0\}$ respectively. 
We assume the following for the generators of quasi-arithmetic means. 
\begin{Ass}\label{ass-f}
Let $\alpha \in \overline{-\mathbb H}$. 
Let $U = U_{\alpha}$ be a simply connected domain containing $\overline{\mathbb H} \setminus \{\alpha\}$. 
Let $f : U \to \mathbb C$ be an  
injective holomorphic  function such that $f\left(\overline{\mathbb H}  \setminus \{\alpha\}\right)$ is convex. 
\end{Ass}

We remark that $\overline{\mathbb H} \setminus \{\alpha\} = \overline{\mathbb H}$ if $\alpha \notin \mathbb{R}$. 
We see that $f^{-1} : f(U) \to U$ is also holomorphic. 
We consider the {\it quasi-arithmetic mean with generator $f$}, or simply the {\it $f$-mean}, which is defined by  
\begin{equation}\label{f-mean}
f^{-1} \left( \frac{1}{n} \sum_{j=1}^{n} f(x_j)  \right) \in \overline{\mathbb H}  \setminus \{\alpha\}, \ \ \  x_1, \cdots, x_n \in \overline{\mathbb H} \setminus \{\alpha\}. 
\end{equation} 
By the assumption that $f\left(\overline{\mathbb H}  \setminus \{\alpha\}\right)$ is convex, 
$$\frac{1}{n} \sum_{j=1}^{n} f(x_j) \in f\left(\overline{\mathbb H}  \setminus \{\alpha\}\right).$$

Throughout this paper, we define the logarithm as follows. 
For $z = r \exp(i \theta)$ where $r > 0$ and $-\pi/2 \le \theta < 3\pi/2$, 
\[ \log z := \log r + i \theta.  \]
As a function of $z$, this is holomorphic on $\mathbb{C} \setminus \left\{z \in \mathbb{C} | \textup{Re}(z) = 0, \textup{Im}(z) \le 0 \right\}$. 
Then, 
\[ \log x = \log |x| + i \pi \mathbf{1}_{(-\infty, 0)}(x), \ \ x \in \mathbb{R} \setminus \{0\}, \]
where $\mathbf{1}_{(-\infty, 0)}$ denotes the indicator function of $(-\infty,0)$. 

For $p \in \mathbb{R}$, 
we let 
\[ z^{p} = \exp(p \log z), \ \ \ z \in \mathbb{C} \setminus \{z \in \mathbb{C} \  | \  \textup{Re}(z) = 0, \textup{Im}(z) \le 0 \}. \]
Then, $|z|^{p} = |z^{p}|$
and 
$$x^{p} = \exp\left(p (\log |x| + i \pi)\right) = (-x)^{p} \exp(i \pi p), \ \ x < 0. $$ 

We remark that $x^{1/(2n+1)}$ is {\it not} a real number if $x < 0$. 
In general, if at least one of $x_1, \cdots, x_n$ is strictly negative, then, it can happen that $\prod_{j=1}^{n} x_j^{1/n} \ne \left(\prod_{j=1}^{n} x_j\right)^{1/n}$.
On the other hand, if all $x_1, \cdots, x_n$ are positive, then, $\prod_{j=1}^{n} x_j^{1/n}  = \left(\prod_{j=1}^{n} x_j\right)^{1/n}$.
We denote  the geometric mean of $x_1, \cdots, x_n$ by $\prod_{j=1}^{n} x_j^{1/n} $, {\it not} by $ \left(\prod_{j=1}^{n} x_j\right)^{1/n}$.

\begin{Exa}\label{power}
(i) Let $\alpha \in \overline{\mathbb H}$. 
If $ f(x) = \log (x+\alpha)$,
then, 
\[ f^{-1} \left( \frac{1}{n} \sum_{j=1}^{n} f(x_j) \right) = \prod_{j=1}^{n} (x_j + \alpha)^{1/n} - \alpha. \]
If $\alpha = 0$, then this is the geometric mean. 
For $\alpha \in \mathbb R$, 
we see that 
$$f\left(\overline{\mathbb H}  \setminus \{-\alpha\}\right) = \left\{ x + yi : x \in \mathbb{R}, y \in [0, \pi] \right\}.$$ 
We need some arguments to show that $f\left(\overline{\mathbb H}  \setminus \{-\alpha\}  \right) = f(\overline{\mathbb H})$ is convex for $\alpha \in \mathbb H$.
See Lemma \ref{lem:cvx} below. 

(ii) Let $\alpha \in \mathbb H$. 
If $f(x) = 1/(x + \alpha)$,
then, 
\begin{equation}\label{Mobius-generator}
f^{-1} \left( \frac{1}{n} \sum_{j=1}^{n} f(x_j) \right) = \frac{\sum_{j=1}^{n} x_j/(x_j + \alpha)}{\sum_{j=1}^{n} 1/(x_j + \alpha)}. 
\end{equation} 
By the circle-to-circle correspondences of the linear fractional transformations, 
$$f\left(\overline{\mathbb H}  \setminus \{-\alpha\}  \right) = f(\overline{\mathbb H}) = \left\{z \in \mathbb C : \left|z+ \frac{i}{2 \textup{Im}(\alpha)}\right| \le \frac{1}{2\textup{Im}(\alpha)} \right\} \setminus \{0\}.$$

If $\alpha \in \mathbb{R}$, 
then this is formally the harmonic mean. 
However, in such case, $f\left(\overline{\mathbb H}  \setminus \{-\alpha\}  \right) = \overline{-\mathbb H} \setminus \{0\}$ is {\it not} convex. 
This case is hard to deal with. 
For example, the harmonic mean of $-\alpha + 1$ and $-\alpha-1$ cannot be defined. 
Therefore in this case  we assume that $\alpha \in \mathbb H$, {\it not} $\alpha \in \overline{\mathbb H}$.  

We remark that 
if $f(x) = (x + \overline{\alpha})/(x +\alpha)$, $\alpha \in \mathbb H$, 
then, \eqref{Mobius-generator} holds 
and 
$$f\left(\overline{\mathbb H}  \setminus \{-\alpha\}  \right) = f(\overline{\mathbb H}) =  \left\{z \in \mathbb{C}  : |z| \le 1 \right\} \setminus \{1\}.$$
This is easy to handle, 
because the value of $f(x)$ is in the unit circle. 
Thus, we see that different forms of generators could have the same quasi-arithmetic means. 
\end{Exa}

\begin{Rem}
(i) The quasi-arithmetic mean with a generator $f$ is {\it homogeneous} if 
\[ f^{-1} \left( \frac{1}{n} \sum_{j=1}^{n} f(a x_j)  \right) = a f^{-1} \left( \frac{1}{n} \sum_{j=1}^{n} f(x_j) \right), \]
for every $a, x_1, \cdots, x_n > 0$.  
By de Finetti-Jessen-Nagumo's result, 
if a generator $f$ of the quasi-arithmetic mean is continuous and homogeneous, then $f(x) = x^p, p \ne 0$ or $f(x) = \log x$. 
See Hardy-Littlewood-P\'olya \cite[p68]{Hardy1952}. \\
(ii) We discuss a form of the averaging property for the quasi-arithmetic means. 
If $f(x) = \log x$, then, for every $x_1, \cdots, x_n \in \mathbb{R}$, 
\begin{equation}\label{qam-minmax} 
\min_{1 \le i \le n} |x_i| \le  \left| f^{-1} \left( \frac{1}{n} \sum_{j=1}^{n} f(x_j) \right) \right| \le \max_{1 \le i \le n} |x_i|. 
\end{equation} 
If $f(x) = x^{p}, \ p \in (-1,1) \setminus \{0\}$, 
then, for each $n$, \eqref{qam-minmax} could fail. 
If $x_i = 1, i \le n-1$ and $x_{n} = -1$, then, 
\[ \left| f^{-1} \left( \frac{1}{n} \sum_{j=1}^{n} f(x_j) \right) \right| = \left| \frac{n -1 + \exp(i \pi p)}{n} \right|^{1/p} < 1 = \min_{1 \le i \le n} |x_i|. \]
On the other hand, we see that for $x_1, \cdots, x_n \in \mathbb{R}$, 
\begin{equation}\label{qam-max}  
\left| f^{-1} \left( \frac{1}{n} \sum_{j=1}^{n} f(x_j) \right) \right| \le \max_{1 \le i \le n} |x_i|.  
\end{equation}  
If $ f(x) = 1/(x + \alpha), \ \alpha \in \mathbb{H}$, 
then, \eqref{qam-max} fails. 
If $x_{2i-1} = -x_{2i} = b > 0, i \ge 1$, then, 
\[ \left|\sum_{i=1}^{2n} \dfrac{x_i}{x_i + \alpha }\right| = n \left| \dfrac{b}{b + \alpha} + \dfrac{b}{b - \alpha} \right| = \frac{2n b^2}{|b^2 - \alpha^2|}, \]
and, 
\[ \left|\sum_{i=1}^{2n} \dfrac{1}{x_i + \alpha}\right| = n \left| \dfrac{1}{b + \alpha} - \dfrac{1}{b - \alpha} \right| = \frac{2 n |\alpha|}{|b^2 - \alpha^2|}.  \]
Hence \eqref{qam-max} fails for sufficiently large $b$. 
Furthermore, by taking sufficiently small $b > 0$  in the above, 
we see that it can happen that 
\[ \left| f^{-1} \left( \frac{1}{n} \sum_{j=1}^{n} f(x_j) \right) \right| < \min_{1 \le i \le n} |x_i|. \]\\
(iii) We call $f^{-1}\left(\prod_{j=1}^{n} f(x_j)^{1/n}\right)$
 the $f$-geometric mean of $x_1, \cdots, x_n$. 
Let $g(x) := \log (f(x))$. 
Then, it is $g$-arithmetic mean, that is, 
\[ f^{-1}\left(\prod_{j=1}^{n} f(x_j)^{1/n} \right) = g^{-1} \left( \frac{1}{n} \sum_{j=1}^{n} g(x_j)  \right).\]
Thus considerations for quasi-geometric means are attributed to those for quasi-arithmetic means.  
\end{Rem}

Hereafter we deal with quasi-arithmetic means of random variables. 
We say that a random variable $X$ is in $L^{\alpha}, \alpha > 0$ if $E\left[\left|X^{\alpha} \right|\right] < +\infty$, and that $X$ is in $L^{\alpha+}$ if $X$ is in $L^{\beta}$ for some $\beta > \alpha$. 
Let the expectation of a  complex-valued random variable $Y$ be 
\[ E[Y] := E\left[\textup{Re}(Y)\right] + i E\left[\textup{Im}(Y)\right].  \]
Let the variance of a  complex-valued random variable $Y$ be 
\[ \textup{Var}(Y) = E\left[ \left|Y - E[Y]\right|^2 \right]. \]
For real-valued random variables, this definition is equal to the usual definition of the variance. 
We see that 
\[ \textup{Var}(Y) = \textup{Var}\left(\textup{Re}(Y)\right) + \textup{Var}\left(\textup{Im}(Y)\right). \]

The following is a basic result for quasi-arithmetic means of i.i.d. random variables.   
\begin{Thm}[law of large numbers]\label{SLLN}
Suppose that $f$ satisfies Assumption \ref{ass-f}. 
Let $(X_i)_i$ be continuous  i.i.d. random variables.  
If $E\left[|f(X_1)| \right] < +\infty$, then, the following convergence  holds almost surely(a.s.): 
\[ \lim_{n \to \infty} f^{-1} \left(\frac{1}{n} \sum_{j=1}^{n} f(X_j)  \right) = f^{-1}\left(E[f(X)] \right). \]
\end{Thm}

Since $f^{-1}$ is continuous, it suffices to show that 
\[ \lim_{n \to \infty} \frac{1}{n} \sum_{j=1}^{n} f(X_j) = E[f(X)], \  \textup{ a.s.,}  \]
which follows from the (standard) strong law of large numbers. 

We remark that even if $E\left[\left|f(X_1)\right|\right] < +\infty$, 
$ f^{-1} \left(\frac{1}{n} \sum_{j=1}^{n} f(X_j)  \right)$ is not necessarily in $L^1$, 
so we cannot consider the $L^1$ convergence in such cases. 

By the delta method, we see that 
\begin{Thm}[central limit theorem]\label{CLT}
Suppose that $f$ satisfies Assumption \ref{ass-f}. 
Let $(X_i)_i$ be continuous  i.i.d. random variables. 
We identify $\mathbb{C}$ with $\mathbb{R}^2$. 
If $E\left[|f(X_1)|^2\right] < +\infty$, then, 
as $n \to \infty$, 
\[ \sqrt{n} \left(   f^{-1} \left( \frac{1}{n} \sum_{j=1}^{n} f(X_j) \right) - f^{-1}(E[f(X_1)])\right) \Rightarrow N\left(0, J(f^{-1}) \textup{Cov}(f(X_1)) J(f^{-1})^{\prime} \right),  \]
where $\Rightarrow$ means the convergence in distribution, 
$N(\cdot, \cdot)$ denotes the two-dimensional normal distribution, 
$J (f^{-1})$ is the Jacobi matrix of $f^{-1}$ at $E[f(X_1)]$, 
that is, 
\[ J (f^{-1}) = \begin{pmatrix} \frac{\partial \textup{Re}(f^{-1})}{\partial x}(E[f(X_1)]) & \frac{\partial \textup{Re}(f^{-1})}{\partial y} (E[f(X_1)])\\ \frac{\partial \textup{Im}(f^{-1})}{\partial x} (E[f(X_1)]) & \frac{\partial \textup{Im}(f^{-1})}{\partial y} (E[f(X_1)]) \end{pmatrix}, \]
$J(f^{-1})^{\prime}$ is the transpose of $J(f^{-1})$, 
and $\textup{Cov}\left(f(X_1)\right)$ is the covariance matrix of the $\mathbb{R}^2$-valued random variable $f(X_1)$, that is, 
\begin{equation}\label{cov-def} 
\textup{Cov}\left(f(X_1)\right) =  \begin{pmatrix} \textup{Var}(\textup{Re}(f(X_1)))  & \textup{Cov}\left(\textup{Re}(f(X_1)), \textup{Im}(f(X_1))\right) \\ \textup{Cov}\left(\textup{Re}(f(X_1)), \textup{Im}(f(X_1))\right)  & \textup{Var}(\textup{Im}(f(X_1))) \end{pmatrix}. 
\end{equation}
\end{Thm}

By this theorem, 
we expect that under some conditions for $X_1$ and $f$,  
\begin{align}\label{Var-WTS}
\lim_{n \to \infty} n \textup{Var} \left( f^{-1} \left( \frac{1}{n} \sum_{j=1}^{n} f(X_j) \right) \right) &= \textup{Var} (f(X_1)) \left| (f^{-1})^{\prime} \left( f^{-1}(E[f(X_1)]) \right) \right|^2 \notag\\ 
&= \frac{\textup{Var} (f(X_1))}{\left| f^{\prime} \left( f^{-1}(E[f(X_1)]) \right) \right|^2}. 
\end{align} 
However we need somewhat more delicate arguments to show this convergence. 
One of the main purposes of this paper is establishing \eqref{Var-WTS} for the two cases in Example \ref{power}. 
This could be regarded as a complex-valued version of \cite{Pakes1999}.

\begin{Def}\label{def-estimators}
For i.i.d. continuous random variables $\{X_i\}_i$, we let\\
(i) (geometric mean)
\[ G^{(\alpha)}_n := \prod_{j=1}^{n} (X_j + \alpha)^{1/n}  - \alpha, \ \alpha \in \overline{\mathbb{H}}. \]
Let $G_n := G_n^{(0)}$.\\
(ii) (M\"obius transform)
\[ C^{(\alpha)}_n := \frac{\sum_{j=1}^{n} X_j/(X_j + \alpha)}{\sum_{j=1}^{n} 1/(X_j + \alpha)}, \ \alpha \in \mathbb{H}. \]
\end{Def}

We will discuss the case that $f(x) = (x + \alpha)^p, 0 < |p| < 1$, in a forthcoming paper. 
It is natural and interesting to consider \eqref{Var-WTS} in a more general framework such as the Bajraktarevi\'c means and the Fr\'echet means, 
but we do not deal with the issue here. 
See \cite{Barczy2021} for more details of the Bajraktarevi\'c means and see \cite{Bhattacharya2002, Bhattacharya2003, Bhattacharya2005} for more details of the Fr\'echet means.

Throughout this paper, we always assume that $(X_i)_{i}$ are i.i.d. continuous random variables. 
For two infinite sequence $(a_n)_n$ and $(b_n)_n$, 
$a_n = O(b_n)$ means that there exists a constant $C > 0$ such that $|a_n| \le C|b_n|$ for every $n$. 

This paper is organized as follows. 
In Section 2, we deal with asymptotic behaviors of the variances for the geometric means $\left(G^{(\alpha)}_n\right)_n$.
Section 3 is devoted to asymptotic behaviors of the variances for $\left(C^{(\alpha)}_n\right)_n$, which are the means generated by some M\"obius transforms. 
In Section 4, by applying the results in Sections 2 and 3,  
we give unbiased strongly consistent point estimators of the Cauchy location and scale parameters simultaneously.

\section{Limit theorems for geometric means}

In this section we consider the case that $f(x) = \log (x+\alpha)$ in Example \ref{power} (i). 
We first establish \eqref{Var-WTS} for $\left(G^{(\alpha)}_n\right)_n$. 
We remark that 
\[  \frac{\textup{Var} (f(X_1))}{\left| f^{\prime} \left( f^{-1}(E[f(X_1)]) \right) \right|^2} = \exp(2E[\log |X_1 + \alpha|])   \textup{Var}\left(\log (X_1 + \alpha) \right). \]

\begin{Thm}[asymptotic behaviors of the variances for geometric mean]\label{var-conv}
If $X_1$ is in $L^{0+}$, then, 
\[ \lim_{n \to \infty}  n \textup{Var}\left( G^{(\alpha)}_n \right) = \exp(2E[\log |X_1 + \alpha|])   \textup{Var}\left(\log (X_1 + \alpha) \right). \]
\end{Thm}

This is an extension of \cite[Theorem 5]{Pakes1999}. 
However the proof of \cite[Theorem 5]{Pakes1999} is not applicable to the complex-valued random variables. 
Our proof is different from the one in \cite[Theorem 5]{Pakes1999}. 
We will repeatedly use l'Hospital's theorem. 

\begin{proof}
For simplicity, we give a proof for the case that $\alpha = 0$. 
We can show the assertion for  general $\alpha$ in the same manner. 

\begin{Lem}\label{lem-1}
Assume that a random variable $X$ is in $L^{0+}$. 
Then, 
\begin{equation}\label{absolute-1} 
\lim_{x \to +0} E\left[|X|^{x}\right]^{1/x} = \exp\left(E[\log|X|]\right),  
\end{equation} 
and
\begin{equation}\label{non-absolute-1} 
\lim_{x \to +0} E\left[X^{x}\right]^{1/x} = \exp\left(E[\log X]\right). 
\end{equation} 
\end{Lem}

\begin{proof}[Proof of Lemma \ref{lem-1}]
Assume that $X$ is in $L^{\delta}$, $\delta > 0$.  
Let 
$$ F(z) := E[X^z], \ z \in \mathbb{C},   |z| < \delta. $$
Then this is holomorphic and $F(0) = 1$. 
Indeed, we see that 
\[ \frac{F(z+h) - F(z)}{h} - E[X^z \log X] = E\left[X^z \left(\frac{X^h - 1}{h} - \log X \right)\right]. \]
If $0 < |h| < 1$, then, 
\[ \left| \frac{X^h - 1}{h} - \log X \right|  \le \frac{|h| |\log X|^2}{2} \exp( |h \log X|) \le C |h| |X|^{|h|} |\log X|^{2} \]
Therefore if $|z| < \delta$, then, we see that 
\[ \left| E\left[X^z \left(\frac{X^h - 1}{h} - \log X \right)\right] \right|  \le C_z |h|  E[ |X|^{|z| + |h|}  |\log X|^{2}] = O(h), \]
where $C_z$ is a constant depending on $z$. 
Therefore, 
\[ \lim_{h \to 0} E\left[X^z \left(\frac{X^h - 1}{h} - \log X \right)\right] = 0,  \]
which implies that $F^{\prime}(z) = E[X^z \log X]$.

Thus we see that 
\[ \lim_{x \to +0} E[X^{x}]^{1/x} = \lim_{x \to +0} \exp\left(\frac{\log F(x)}{x}\right) = \exp\left(E[\log X]\right).  \] 
We now have \eqref{non-absolute-1}. 
\end{proof}

We remark that 
\[ |\exp(E[\log X])| = \exp(E[\log|X|]). \]
By this, Lemma \ref{lem-1}, and l'Hospital's theorem, it holds that 
\begin{equation}\label{1st-reduction} 
\lim_{x \to +0} \frac{E\left[|X|^{2x}\right]^{1/x} - \left|E[X^x]\right|^{2/x}}{x} = \lim_{x \to +0} \dfrac{d}{dx} E\left[|X|^{2x}\right]^{1/x} 
-  \lim_{x \to +0} \dfrac{d}{dx} \left|E\left[X^{x}\right]\right|^{2/x},  
\end{equation} 
provided that the limits in the right hand side of the above  display  exists. 

(1) We compute $ \lim_{x \to +0} \dfrac{d}{dx} E\left[|X|^{2x}\right]^{1/x}$. 
Let 
\[ G(x) := E\left[|X|^{2x}\right]  \]
and  
\[ g(x) := \frac{\log G(x)}{x}, \ x > 0.  \]
Then, 
\[ \dfrac{d}{dx} E\left[|X|^{2x}\right]^{1/x} =  G(x)^{1/x} g^{\prime}(x).   \]
By this and \eqref{absolute-1}, it holds that 
\begin{equation}\label{reduction-g} 
\lim_{x \to +0} \dfrac{d}{dx} E\left[|X|^{2x}\right]^{1/x} = \exp(2E[\log|X|]) \lim_{x \to +0} g^{\prime}(x), 
\end{equation}
provided that the limit in the right hand side of the above  display  exists. 
Since 
\[ g^{\prime}(x) = \frac{x G^{\prime}(x) -  G(x) \log G(x)}{x^2 G(x)} \]
and 
\[ \lim_{x \to +0} G(x) = 1, \]
we see that by l'Hospital's theorem, 
\begin{equation}\label{reduction-G} 
\lim_{x \to +0} g^{\prime}(x) = \lim_{x \to +0}  \frac{\frac{d}{dx}  \left( x G^{\prime}(x) -  G(x) \log G(x) \right)}{\frac{d}{dx}  \left( x^2 G(x) \right)}, 
\end{equation} 
provided that the limit in the right hand side of the above  display  exists. 

Since 
\[ \frac{d}{dx}  \left( x G^{\prime}(x) -  G(x) \log G(x) \right) = x \left(G^{\prime\prime}(x)  - G^{\prime}(x)g(x) \right) \]
and 
\[ \frac{d}{dx}  \left( x^2 G(x) \right) = 2x G(x)  + x^2 G^{\prime}(x), \]
we see that for $x > 0$, 
\begin{equation}\label{reduction-G2}  
\frac{\frac{d}{dx}  \left( x G^{\prime}(x) -  G(x) \log G(x) \right)}{\frac{d}{dx}  \left( x^2 G(x) \right)} = \frac{G^{\prime\prime}(x)  - G^{\prime}(x)g(x)}{2G(x)  + x G^{\prime}(x)}. 
\end{equation} 

In the same manner as in the proof of Lemma \ref{lem-1}, we see that 
\begin{equation}\label{G1}  
\lim_{x \to +0} G^{\prime}(x) = 2  E\left[ \log |X| \right],  
\end{equation} 
and 
\begin{equation}\label{G2} 
\lim_{x \to +0} G^{\prime\prime}(x)  = 4  E\left[ (\log |X|)^2 \right]. 
\end{equation} 
By l'Hospital's theorem and Lemma \ref{lem-1}, 
we see that 
\begin{equation}\label{g0}  
\lim_{x \to +0} g(x) = \frac{G^{\prime}(0)}{G(0)} = 2  E\left[ \log |X| \right]. 
\end{equation}

By \eqref{G1}, \eqref{G2} and \eqref{g0}, we see that 
\[ \lim_{x \to +0} \frac{G^{\prime\prime}(x)  - G^{\prime}(x)g(x)}{2G(x)  + x G^{\prime}(x)}  = 2 (E\left[ (\log |X|)^2 \right] -  E\left[ \log |X| \right]^2 ) = 2 \textup{Var}(\log|X|). \]

By this, \eqref{reduction-g}, \eqref{reduction-G} and \eqref{reduction-G2}, we see that 
\begin{equation}\label{ans-1} 
\lim_{x \to +0} \dfrac{d}{dx} E\left[|X|^{2x}\right]^{1/x} = 2 \exp(2E[\log|X|]) \textup{Var}(\log|X|).  
\end{equation}

(2) We consider $ \lim_{x \to +0} \dfrac{d}{dx} \left|E\left[X^{x}\right]\right|^{2/x}$. 
Let 
\[ H(x) := |E\left[X^{x}\right]|^{2} \]
and 
\[ h(x) := \frac{\log H(x)}{x}, \ x > 0.  \]
Then, 
\[ \dfrac{d}{dx} H(x)^{1/x} =  H(x)^{1/x} h^{\prime}(x).   \]

By \eqref{non-absolute-1}, we see that 
\[ \lim_{x \to +0} \frac{\log E[X^x]}{x} = E[\log X]. \]
By this and 
\[ E[\log X] + \overline{E[\log X]} = 2 \textup{Re}(E[\log X]) = 2E[\log |X|],  \]
we see that 
\[ \lim_{x \to +0}  H(x)^{1/x}  = \exp(2E[\log |X|]).  \]

Hence, 
\begin{equation}\label{reduction-h}
\lim_{x \to +0}  \dfrac{d}{dx} H(x)^{1/x} =  \exp(2E[\log |X|]) \lim_{x \to +0}  h^{\prime}(x), 
\end{equation}  
provided that the limit in the right hand side of the above display  exists. 

Since 
\[ h^{\prime} (x) = \frac{x H^{\prime}(x) -  H(x) \log H(x)}{x^2 H(x) }  \]
and 
\[ \lim_{x \to +0} H(x) = 1, \]
we see that by l'Hospital's theorem, 
\begin{equation}\label{reduction-H} 
\lim_{x \to +0} h^{\prime}(x) = \lim_{x \to +0}  \frac{\frac{d}{dx}  \left( x H^{\prime}(x) -  H(x) \log H(x) \right)}{\frac{d}{dx}  \left( x^2 H(x) \right)}, 
\end{equation} 
provided that the limit in the right hand side of the above  display  exists. 

We see that for $x > 0$, 
\begin{equation}\label{reduction-H2}  
\frac{\frac{d}{dx}  \left( x H^{\prime}(x) -  H(x) \log H(x) \right)}{\frac{d}{dx}  \left( x^2 H(x) \right)} = \frac{H^{\prime\prime}(x)  - H^{\prime}(x)h(x)}{2H(x)  + x H^{\prime}(x)}. 
\end{equation} 

In the same manner as in the proof of Lemma \ref{lem-1}, we see that 
\begin{equation}\label{H1}  
\lim_{x \to +0} H^{\prime} (x)  = 2  E\left[ \log |X| \right], 
\end{equation} 
and 
\begin{equation}\label{H2} 
\lim_{x \to +0}  H^{\prime\prime} (x)  = 2 \textup{Re}(E\left[ (\log X)^2 \right]) + 2 |E[\log X]|^2. 
\end{equation}

By l'Hospital's theorem and Lemma \ref{lem-1},  we see that 
\begin{equation}\label{h0}  
\lim_{x \to +0} h(x) = \frac{H^{\prime}(0)}{H(0)} = 2  |E\left[ \log X \right]|. 
\end{equation} 

By \eqref{reduction-H} - \eqref{h0} and Remark \ref{var-conv-rem} below, 
we see that 
\begin{align*} 
\lim_{x \to +0} h^{\prime} (x) &= \frac{2 \textup{Re}(E\left[ (\log X)^2 \right]) + 2 |E[\log X]|^2 - \left(2  E\left[ \log |X| \right] \right)^2}{2} \\
&= \textup{Var}(\log |X|) + E\left[\left(\log \frac{X}{|X|}\right)^2\right] - \left( E\left[\log \frac{X}{|X|}\right] \right)^2 \\
&= \textup{Var}(\textup{Re}(\log X)) - \textup{Var}(\textup{Im}(\log X)). 
\end{align*} 

By this and \eqref{reduction-h}, we see that 
\[ \lim_{x \to +0} \dfrac{d}{dx} \left|E\left[X^{x}\right]\right|^{2/x} =  \exp(2E[\log |X|])   \left(\textup{Var}(\textup{Re}(\log X)) - \textup{Var}(\textup{Im}(\log X))\right). \]

By this, \eqref{1st-reduction} and \eqref{ans-1},  
we see that 
\begin{align*} 
\lim_{x \to +0} \frac{E\left[|X|^{2x}\right]^{1/x} - \left|E[X^x]\right|^{2/x}}{x} &= \exp(2E[\log |X|])\left( \textup{Var}(\textup{Re}(\log X)) + \textup{Var}(\textup{Im}(\log X)) \right) \\
&= \exp(2E[\log |X|]) \textup{Var}(\log X). 
\end{align*}

By using this and 
\[  n \textup{Var}\left( G_n \right) =  \frac{E\left[|X|^{2/n}\right]^{n} - \left|E[X^{1/n}]\right|^{2n}}{1/n}, \]
we have Theorem \ref{var-conv}. 
\end{proof}

\begin{Rem}\label{var-conv-rem}
By computation, we see that 
\begin{align*}  
\textup{Var}(\log X_1) &= \textup{Var}(\textup{Re}(\log X_1)) + \textup{Var}(\textup{Im}(\log X_1)) \\
&= 2E\left[(\log |X_1|)^2 \right] - \textup{Re}(E[(\log X_1)^2]) - |E[\log X_1]|^2.  
\end{align*} 
\end{Rem}

\begin{Lem}\label{lem:cvx}
Let $f(x) = \log(x+\alpha), \alpha \in \mathbb H$.
Then, $f(\overline{\mathbb H})$ is convex. 
\end{Lem}

\begin{proof}
Let $z_1, z_2 \in f(\overline{\mathbb H})$ and $t \in [0,1]$. 
Then, 
$$f^{-1}((1-t)z_1 + tz_2) = \exp((1-t)z_1 + tz_2) - \alpha.$$
Hence, 
$$\textup{Im}\left(f^{-1}((1-t)z_1 + tz_2)\right) = \exp\left((1-t)\textup{Re}(z_1) + t\textup{Re}(z_2)\right) \sin((1-t)\textup{Im}(z_1) + t\textup{Im}(z_2)) - \textup{Im}(\alpha).$$
Hence, 
$f^{-1}((1-t)z_1 + tz_2) \in \overline{\mathbb H}$ if and only if 
\begin{equation}\label{eq:log-im} 
(1-t)\textup{Re}(z_1) + t\textup{Re}(z_2) + \log (\sin((1-t)\textup{Im}(z_1) + t\textup{Im}(z_2))) \ge \log(\textup{Im}(\alpha)). 
\end{equation} 
Since $\log(\sin(x))$ is strictly concave on $x \in (0, \pi)$, 
\[ (1-t)\textup{Re}(z_1) + t\textup{Re}(z_2) + \log (\sin((1-t)\textup{Im}(z_1) + t\textup{Im}(z_2))) \]
\[\ge (1-t)\left(\textup{Re}(z_1)+ \log\sin(\textup{Im}(z_1))\right) + t\left(\textup{Re}(z_2)+ \log\sin(\textup{Im}(z_2))\right). \]
Since $z_1, z_2 \in f(\overline{\mathbb H})$,
\[ \textup{Re}(z_j)+ \log\sin(\textup{Im}(z_j)) \ge \log(\textup{Im}(\alpha)), \ j = 1,2.\]
We now have \eqref{eq:log-im}. 
\end{proof}

Now we prepare a non-random lemma. 
\begin{Prop}\label{im-part}
(i) Let $\alpha \in \mathbb{R}$. 
Then it holds that all of the signs of $x_i + \alpha, 1 \le i \le n$, are equal to each other if and only if 
\begin{equation}\label{imzero}
\textup{Im}\left( f^{-1} \left( \frac{1}{n} \sum_{j=1}^{n} f(x_j) \right) \right)= 0. 
\end{equation}
(ii) If $\alpha \in \mathbb{H}$, then \eqref{imzero} holds if and only if $x_1 = \cdots = x_n$. 
\end{Prop}

We recall that fact that $P(X = Y) = 0$ for every independent continuous random variables $X$ and $Y$. 
Then we see that 
\begin{Cor}
(i) If $\alpha \in \mathbb{R}$ and $X_1$ satisfies that 
$$P(X_1 > M) P(X_1 < -M) > 0$$ 
for every $M > 0$, 
then $\textup{Im}\left(G_n^{(\alpha)}\right) = 0$ with positive probability. \\
(ii) If $\alpha \in \mathbb{H}$, then $\textup{Im}\left(G_n^{(\alpha)}\right) > 0$ a.s. 
\end{Cor}

As we will see in Section 6 below, 
it could happen that $\textup{Im}(G_n) = 0$ with positive probability, if $X_1$ follows the Cauchy distribution. 
This is not satisfactory because the imaginary part infers the scale parameter of the Cauchy distribution.

\begin{proof}[Proof of Proposition \ref{im-part}]
(i) We see that 
\[  \textup{Im}\left( \prod_{i=1}^{n} (x_i + \alpha)^{1/n} -\alpha \right) = \prod_{i=1}^{n} |x_i + \alpha|^{1/n} \sin\left( \frac{\ell}{n} \pi \right), \]
where $\ell$ is the number of $i$ such that $x_i + \alpha < 0$. 
Hence,  \[ \textup{Im}\left( \prod_{i=1}^{n} (x_i + \alpha)^{1/n} -\alpha \right) = 0 \] 
if and only if $\ell = 0 \textup{ or } n$, which means that all of the signs of $x_i + \alpha, 1 \le i \le n$, are equal to each other.

(ii) The proof is similar to the one of Lemma \ref{lem:cvx}. 
Without loss of generality, we can assume that $\textup{Im}(\alpha) = 1$. 
Let $x_j + \alpha = r_j \exp(i \theta_j), r_j > 0, 0 < \theta_j < \pi$ for each $j$.  
Then, by using the fact that $\sin \theta_j = 1/r_j$, 
we see that 
\begin{align*} 
\textup{Im}\left( \prod_{j=1}^{n} (x_j + \alpha)^{1/n} \right) &= \left( \prod_{j=1}^{n} r_j^{1/n}  \right) \sin\left( \frac{1}{n} \sum_{j=1}^{n} \theta_j \right) \\
&= \exp\left( \log\left(\sin\left( \frac{1}{n} \sum_{j=1}^{n} \theta_j \right)\right) -  \frac{1}{n} \sum_{j=1}^{n} \log (\sin \theta_j)  \right). 
\end{align*}

Hence, 
$$ \textup{Im}\left( \prod_{j=1}^{n} (x_j + \alpha)^{1/n} \right) > \textup{Im}(\alpha) = 1$$ 
if and only if 
\begin{equation}\label{logsin} 
\log\left(\sin\left( \frac{1}{n} \sum_{j=1}^{n} \theta_j \right)\right) >  \frac{1}{n} \sum_{j=1}^{n} \log (\sin \theta_j).  
\end{equation}
Since $\log(\sin(x))$ is strictly concave on $x \in (0, \pi)$, 
\eqref{logsin} holds if and only if it fails that $x_1 = \cdots = x_n$. 
\end{proof}

\section{Limit theorems for quasi-arithmetic means by M\"obius transformations}

In this section, we consider the case that the generator $f$ is a M\"obius transformation as in Example \ref{power} (iii). 
We establish \eqref{Var-WTS} for $\left(C^{(\alpha)}_n\right)_n$.  

Let 
\begin{equation}\label{def-Jk} 
J^{(\alpha)}_n := \frac{1}{n} \sum_{j = 1}^{n} \frac{1}{X_j + \alpha} \textup{ \ and \ } \mu^{(\alpha)} := E\left[ \frac{1}{X_1 + \alpha} \right].
\end{equation} 
Then, 
$C^{(\alpha)}_n = \frac{1}{J^{(\alpha)}_n}  -\alpha$
and 
\begin{equation}\label{unbiased-Jk} 
E\left[J^{(\alpha)}_n \right] = \mu^{(\alpha)}, \ n \ge 1. 
\end{equation}

We see that 
\begin{equation}\label{VarJ} 
\textup{Var}\left(C^{(\alpha)}_n \right) 
= E\left[ \left| \frac{1}{ J^{(\alpha)}_n } -     \frac{1}{\mu^{(\alpha)}}       \right|^2   \right] -  \left|  E\left[ \frac{1}{ J^{(\alpha)}_n }  \right]   -     \frac{1}{\mu^{(\alpha)}}    \right|^2. 
\end{equation} 

We remark that $J^{(\alpha)}_n, \mu^{(\alpha)} \in \overline{-\mathbb H} \setminus \mathbb{R}, n \ge 1$.

\begin{Thm}\label{Var-Ck}
If $X_1$ is in $L^{0+}$, then, 
\[  \lim_{n \to \infty} n \textup{Var}\left(C^{(\alpha)}_n \right) = \frac{1}{\left|\mu^{(\alpha)}\right|^4} \textup{Var}\left( \frac{1}{X_1 + \alpha}\right).  \]
\end{Thm}

The asymptotic behaviors of the variances  for the harmonic mean of real-valued random variables are considered by Pakes \cite[Theorem 7 and Eq. (16)]{Pakes1999}. 
The proof of \cite[Theorem 7]{Pakes1999} depends on the monotonicity of a sequence corresponding to $ \left(E\left[ 1/J^{(\alpha)}_n  \right]\right)_n$, 
however our case is the harmonic mean of complex-valued random variables and we take a different proof using a complex mean-value theorem \cite[Theorem 2.2]{Evard1992}.  

Theorem \ref{Var-Ck} follows from \eqref{VarJ}, Propositions \ref{Ck} and \ref{VarJ-small} below. 

\begin{Prop}\label{Ck}
If $X_1$ is in $L^{0+}$, then, 
\[ \lim_{n \to \infty} n E\left[ \left| \frac{1}{ J^{(\alpha)}_n } -     \frac{1}{\mu^{(\alpha)}}       \right|^2   \right]  
= \frac{1}{\left|\mu^{(\alpha)}\right|^4} \textup{Var}\left( \frac{1}{X_1 + \alpha}\right). \]
\end{Prop}

\begin{proof}
Assume that $E\left[\left|X_1 \right|^{\delta}\right] < +\infty, \delta > 0$. 
Then, 
$E\left[\left|X_1 +\alpha  \right|^{\delta}\right] < +\infty$. 

By the geometric-harmonic mean inequality, we see that 
\[ \left| \frac{1}{J^{(\alpha)}_n} \right|^2 \le \frac{1}{\textup{Im}(\alpha)^2} \frac{n^2}{\left( \sum_{j = 1}^{n} \frac{1}{\left|X_j + \alpha\right|^2} \right)^2} \le \frac{1}{\textup{Im}(\alpha)^2}  \prod_{j=1}^{n} \left|X_j + \alpha\right|^{4/n}. \]
Hence, for every $\eta > 0$, 
\begin{equation}\label{Jk-integrability} 
E\left[ \left| \frac{1}{J^{(\alpha)}_n} \right|^{\eta} \right] \le \frac{1}{\textup{Im}(\alpha)^\eta} E\left[\left| X_1 + \alpha\right|^{\delta} \right]^{2\eta/\delta} < +\infty. 
\end{equation}

Let $A_n$ be the random variable such that 
\begin{equation}\label{def-Ak} 
\frac{1}{J^{(\alpha)}_n} - \frac{1}{\mu^{(\alpha)}} = -\frac{1}{(\mu^{(\alpha)})^2} \left( J^{(\alpha)}_n - \mu^{(\alpha)} \right) + A_n.  
\end{equation} 
Then, 
\[ \left|  \frac{1}{J^{(\alpha)}_n} - \frac{1}{\mu^{(\alpha)}}  \right|^2 = \frac{\left| J^{(\alpha)}_n - \mu^{(\alpha)}  \right|^2}{|\mu^{(\alpha)}|^4} + |A_n|^2 - 2 \textup{Re} \left(  \frac{\overline{A_n}}{(\mu^{(\alpha)})^2} \left( J^{(\alpha)}_n - \mu^{(\alpha)} \right)\right). \]

Let 
\[ D := \left\{z \in \mathbb{C} : |z - \mu^{(\alpha)}| < \frac{|\mu^{(\alpha)}|}{2}\right\} \subset \overline{-\mathbb H} \setminus \mathbb{R}, \]
which is an open convex subset of $ \overline{-\mathbb H} \setminus \mathbb{R}$. 
Now by the complex Rolle theorem\footnote{This is also known as the Grace-Heawood theorem.} \cite[Theorem 2.2]{Evard1992}, 
we see that if $J^{(\alpha)}_n \in D$, then, 
\begin{equation}\label{Rolle} 
|A_n| \le |J^{(\alpha)}_n - \mu^{(\alpha)}|^2 \frac{64}{\left|\mu^{(\alpha)}\right|^3}  
\end{equation} 

We can show the following estimates in the same manner as in the case that real-valued bounded random variables: 
\begin{Lem}\label{complex-4th}
(i) \[ E\left[   \left|J^{(\alpha)}_n - \mu^{(\alpha)}\right|^2 \right] = \frac{1}{n} \textup{Var}\left(\frac{1}{X_1 + \alpha}\right). \]
(ii) \[ E\left[ \left|J^{(\alpha)}_n - \mu^{(\alpha)} \right|^4 \right] = O\left(\frac{1}{n^2}\right). \]
\end{Lem}

By the H\"older inequality, 
we see that 
\[ E\left[ \left|J^{(\alpha)}_n - \mu^{(\alpha)}\right|^3, \ J^{(\alpha)}_n \in D \right] \]
\[\le E\left[ |J^{(\alpha)}_n - \mu^{(\alpha)}|^4, \ J^{(\alpha)}_n \in D \right]^{1/2} E\left[ |J^{(\alpha)}_n - \mu^{(\alpha)}|^2, \ J^{(\alpha)}_n \in D \right]^{1/2}. \]
By this, \eqref{Rolle} and Lemma \ref{complex-4th}, we see that 
\[ \lim_{n \to \infty} k E\left[|A_n| \cdot \left|J^{(\alpha)}_n - \mu^{(\alpha)}\right|, \ J^{(\alpha)}_n \in D  \right] = 0, \]
and, 
\[ \lim_{n \to \infty} n E\left[|A_n|^2, \ J^{(\alpha)}_n \in D  \right] = 0. \]

Thus we see that 
\[ \lim_{n \to \infty} n E\left[\left|  \frac{1}{J^{(\alpha)}_n} - \frac{1}{\mu^{(\alpha)}}  \right|^2 - \frac{\left| J^{(\alpha)}_n - \mu^{(\alpha)}  \right|^2}{|\mu^{(\alpha)}|^4}, J_n^{(\alpha)} \in D\right]= 0. \]

Since $\left(1/(X_j + \alpha) \right)_j$ is uniformly bounded, 
by the Cramer-Chernoff method,  
we see that  for some $\beta \in (0,1)$, 
\begin{equation}\label{LDP-Jk}
P\left(J_n \notin D\right) = O(\beta^n). 
\end{equation} 

By this and \eqref{Jk-integrability}, 
we see that 
\[ \lim_{n \to \infty} n E\left[\left|  \frac{1}{J^{(\alpha)}_n} - \frac{1}{\mu^{(\alpha)}}  \right|^2 - \frac{\left| J^{(\alpha)}_n - \mu^{(\alpha)}  \right|^2}{|\mu^{(\alpha)}|^4}, J_n^{(\alpha)} \notin D\right]= 0.\]

Therefore,
\[ \lim_{n \to \infty} n E\left[\left|  \frac{1}{J^{(\alpha)}_n} - \frac{1}{\mu^{(\alpha)}}  \right|^2 - \frac{\left| J^{(\alpha)}_n - \mu^{(\alpha)}  \right|^2}{|\mu^{(\alpha)}|^4}\right]= 0.\]
By this and Lemma \ref{complex-4th}, 
we have the assertion. 
\end{proof}

\begin{Prop}\label{VarJ-small} 
If $X_1$ is in $L^{0+}$, then, 
\[ \left| E\left[ \frac{1}{ J^{(\alpha)}_n }\right] - \frac{1}{\mu^{(\alpha)}}  \right| = O\left(\frac{1}{n}\right). \]
In particular, 
\[ \lim_{n \to \infty} n \left| E\left[ \frac{1}{ J^{(\alpha)}_n } \right] -     \frac{1}{\mu^{(\alpha)}} \right|^2 = 0.  \]
\end{Prop}

\begin{proof}
We use notation used for the proof of Proposition \ref{Ck}. 
By \eqref{unbiased-Jk}, \eqref{def-Ak}, \eqref{Rolle} and Lemma \ref{complex-4th}, we see that 
\[  \left| E\left[ \frac{1}{ J^{(\alpha)}_n }  -     \frac{1}{\mu^{(\alpha)}}, \ J^{(\alpha)}_n \in D \right] \right| \le \frac{           \left|  E\left[J^{(\alpha)}_n - \mu^{(\alpha)}\ J^{(\alpha)}_n \notin D\right] \right|           }{\left|\mu^{(\alpha)} \right|^2} + \left| E\left[A_n, \ J^{(\alpha)}_n \in D \right] \right|.  \]

We also see  that 
\begin{align*} \left| E\left[A_n, \ J^{(\alpha)}_n \in D \right] \right| &\le E\left[|A_n|,  \ J^{(\alpha)}_n \in D\right] \\
&\le \frac{64}{\left|\mu^{(\alpha)}\right|^3} E\left[\left|J^{(\alpha)}_n - \mu^{(\alpha)}\right|^2\right] = O\left(\frac{1}{n}\right).  
\end{align*}

By \eqref{LDP-Jk} and \eqref{Jk-integrability}, 
\[ \left|  E\left[J^{(\alpha)}_n - \mu^{(\alpha)}\ J^{(\alpha)}_n \notin D\right] \right| =  O\left(\frac{1}{n}\right) \]
and 
\[ \left| E\left[ \frac{1}{ J^{(\alpha)}_n }  -     \frac{1}{\mu^{(\alpha)}}, \ J^{(\alpha)}_n \notin D \right] \right| =  O\left(\frac{1}{n}\right).  \]

Thus we have the assertion. 
\end{proof} 

\section{Point estimation for parameters of the Cauchy distribution}

The Cauchy distribution is often used to formulate statistical models with heavy-tailed phenomena. 
We cannot define its expected value and its variance, and it has no moment generating function, due to its heavy tails. 
It is a canonical example of the stable distribution. 
It also appears in physics, and is called the Lorentz distribution alternatively.

We briefly review known results for parameter estimations for the location and scale parameters of the Cauchy distribution. 
Various approaches have been taken. 
The maximal likelihood estimation has been considered by \cite{Haas1970, Copas1975, Ferguson1978, Hinkley1978, Gabrielsen1982, Reeds1985, Saleh1985, Bai1987,  Vaughan1992, McCullagh1992, McCullagh1993, McCullagh1996, Mardia1999, Auderset2005, Matsui2005}.  
This is widely used in many parametric statistical models. 
Since it may not have an explicit formula,  
the Newton-Raphson method has often been used.
However, it may diverge for some samples, and, the larger the sample number is, the harder the computation might be. 
For the Cauchy distribution, 
see \cite{Okamura2021} for some computational problems of the Newton-Raphson method. 
The order statistics, which includes analysis for the central values and quantiles of location-scale families,  is used in \cite{Ogawa1962, Ogawa1962a, Rothenberg1964, Barnett1966, Bloch1966, Chan1970, Balmer1974, Cane1974, Rublik2001, Zhang2009, Kravchuk2012}.  
The Bayes approach is taken by  \cite{Howlader1988}. 
The window estimates are used by \cite{Higgins1977, Higgins1978}. 
Other approaches are taken by \cite{Boos1981, Guertler2000, Besbeas2001, Onen2001, Kravchuk2005, CohenFreue2007}.  
The multidimensional case is considered by \cite{Ferguson1962, Arslan1998}. 
The Cauchy distribution is a class of the stable distributions.  
\cite{Press1972}, \cite[Chapter 4]{Zolotarev1986} and, \cite{Matsui2020} consider parameter estimations of the one-dimensional stable distributions by analyzing the characteristic functions. 
In particular, \cite[Chapter 4]{Zolotarev1986} gives strongly consistent, $\sqrt{n}$-consistent and asymptotically unbiased estimators of parameters.  
These results try to balance the computational complexities  with the  consistency, efficiency, and  robustness of estimators. 
The results obtained before 1994 are thoroughly surveyed in the book by Johnson-Kotz-Balakrishnan \cite[Chapter 16]{Johnson1994}. 

In this section, 
we introduce unbiased and closed-form estimators by using quasi-arithmetic means, which are easy to compute and show various convergence results rigorously, 
thanks to the structure of the sum of i.i.d. random variables appearing in \eqref{f-mean}. 
Indeed, by Theorem \ref{CLT}, the quasi-arithmetic means in Definition \ref{def-estimators} 
are $\sqrt{n}$-consistent estimators of $\mu + \sigma i$. 
Furthermore we can construct confidence discs and consider large deviations of the geometric mean. 
They are considered in \cite{Akaoka2021-2, Akaoka2021-3}. 

Denote by $C(\mu, \sigma)$ the Cauchy distribution with location $\mu \in \mathbb{R}$ and scale $\sigma > 0$.  
The density function $p(\cdot; (\mu, \sigma))$ of $C(\mu, \sigma)$ is given by %
\[ p(x; (\mu, \sigma)) = \frac{\sigma}{\pi} \frac{1}{(x-\mu)^2 + \sigma^2}, \ x \in \mathbb{R}. \]
The Fisher information matrix of $C(\mu, \sigma)$ is given by
\[ I(\mu, \sigma) =  \dfrac{1}{2\sigma^2} I_2, \]
where $I_2$ is the identity matrix of degree 2. 
Let $U_n$ be a complex-valued unbiased estimator of $\mu + \sigma i$ consisting of $n$ samples.   
Then, $\textup{Re}(U_n)$ and $\textup{Im}(U_n)$ are unbiased estimators of $\mu$ and $\sigma$ respectively. 
Hence by the Cramer-Rao inequality, 
\begin{equation}\label{CR}
\textup{Var}(U_n) = \textup{Var}\left(\textup{Re}(U_n)\right)  + \textup{Var}\left(\textup{Im}(U_n) \right)  \ge \frac{4\sigma^2}{n}.  
\end{equation} 

The quasi-arithmetic mean has an aspect of an $M$-estimator. 
Let $\psi(x, \mu + \sigma i) := f(x) - f(\mu + \sigma i)$. 
Then, 
$$ f^{-1}\left( \frac{1}{n} \sum_{j=1}^{n} f(X_j)  \right) = \mu + \sigma  i $$ 
if and only if 
$$ \sum_{j=1}^{n} \psi(X_j, \mu + \sigma  i ) = 0.$$
We do not give further discussions of this aspect here. 

{\it Throughout this section, we assume that  $(X_i)_i$ are i.i.d. random variables following the Cauchy distribution $C(\mu, \sigma)$. }  
The following is obtained by applying Theorems \ref{SLLN} and \ref{CLT} to the Cauchy case. 

\begin{Prop}\label{integral}
Let $f$ be a function as in Assumption \ref{ass-f}.  
Then,\\
(i) 
Assume that  
\begin{equation}\label{f-L1} 
E\left[|f(X_1)|\right] < +\infty,
\end{equation} 
\begin{equation}\label{sublinear}
\lim_{R \to \infty} \frac{1}{R} \sup_{|z| = R} |f(z)| = 0, 
\end{equation} 
and 
\begin{equation}\label{radius}
\lim_{\epsilon \to 0} \epsilon \sup_{|z-a| = \epsilon} |f(z)| = 0.
\end{equation} 
Then it holds that 
\begin{equation}\label{unbiased-basic} 
E[f(X_1)] = f(\mu + \sigma  i ) 
\end{equation}
and the following convergence holds almost surely: 
\begin{equation}\label{f-SLLN} 
\lim_{n \to \infty} f^{-1} \left( \frac{1}{n} \sum_{j=1}^{k} f(X_j) \right) = \mu + \sigma  i.  
\end{equation}
(ii) 
In addition to the assumptions in (i), we further assume that  
\begin{equation}\label{f-L2} 
E\left[|f(X_1)|^2 \right] < +\infty. 
\end{equation} 
Then, 
\begin{equation}\label{f-CLT} 
\sqrt{n} \left(   f^{-1} \left( \frac{1}{n} \sum_{j=1}^{n} f(X_j) \right) - \mu + \sigma  i   \right) \Rightarrow N\left(0, \frac{\textup{Var}(f(X_1))}{2 \left| f^{\prime}(\mu + \sigma  i ) \right|^2} I_2 \right),
\end{equation} 
as $n \to \infty.$
\end{Prop}

\begin{proof}
We show (i). 
\eqref{unbiased-basic} follows from the Cauchy integral formula, due to \eqref{sublinear} and \eqref{radius}.  
By the strong law of large numbers, 
\eqref{f-SLLN} follows from \eqref{f-L1} and  \eqref{unbiased-basic}. 

We show (ii). 
By \eqref{f-L2}, we can apply Theorem \ref{CLT}. 
By noting \eqref{unbiased-basic}, 
it suffices to show that 
\begin{equation}\label{matrix-equal}
J(f^{-1}) \textup{Cov}\left(f(X_1)\right) J(f^{-1})^{\prime}  = \frac{\textup{Var}(f(X_1))}{2 \left| f^{\prime}(\mu + \sigma  i ) \right|^2} I_2. 
\end{equation} 
We see that $f(X_1) - f(\mu + \sigma  i )$ is a {\it proper random variable}, that is, 
\[ E\left[ f(X_1) - f(\mu + \sigma  i ) \right] = E\left[ (f(X_1) - f(\mu + \sigma  i ))^ 2 \right] = 0. \]

Hence, 
\[  \textup{Cov}\left( \textup{Re}(f(X_1)), \textup{Im}(f(X_1)) \right) = 0, \]
and 
\begin{equation}\label{re=im} 
\textup{Var}\left( \textup{Re}(f(X_1)) \right) = \textup{Var}\left( \textup{Im}(f(X_1)) \right) = \frac{\textup{Var}(f(X_1))}{2}.  
\end{equation}
Hence, by recalling \eqref{cov-def}, 
\begin{equation*} 
\textup{Cov}\left(f(X_1)\right) =  \frac{\textup{Var}(f(X_1))}{2} I_2.  
\end{equation*}

Since $f^{-1}$ is holomorphic, we see that 
\[ J(f^{-1}) J(f^{-1})^{\prime} = \left|(f^{-1})^{\prime} \left( f(\mu + \sigma  i) \right)\right|^{2} I_2 = \frac{1}{\left|f^{\prime} (\mu + \sigma  i )\right|^{2}}  I_2. \]

Thus we have \eqref{matrix-equal}. 
\end{proof} 

In the following subsections, we show that the functions in Example \ref{power} (i) and (ii) satisfy \eqref{sublinear}, \eqref{radius} and \eqref{f-L2}.    
Hereafter we do not deal with \eqref{f-CLT}, which is attributed to the computation of $\textup{Var}(f(X_1))$ and $f^{\prime}(\mu + \sigma  i )$. 
Instead we will establish \eqref{Var-WTS}, by using the results obtained by Sections 2 and 3.

\subsection{Geometric mean}

In this subsection, we consider the case that $f(x) = \log(x + \alpha), \ \alpha \in \overline{\mathbb{H}}$.

\begin{Thm}[geometric mean]\label{Gk-Cauchy}
Let $\alpha \in \overline{\mathbb H}$. 
Then we have that\\
(i) If $n \ge 2$, then $G_n^{(\alpha)}$ is unbiased.\\
(ii) 
The following convergence holds a.s. and in $L^p$ for every $p$, 
\[ \lim_{n \to \infty} G_n^{(\alpha)} = \mu + i \sigma. \]
(iii) 
\begin{align}\label{Var-Gk-Cauchy} 
\lim_{n \to \infty} n \textup{Var}\left( G_n^{(\alpha)}  \right) &= 2 \left((\mu + \textup{Re}(\alpha))^2 + (\sigma + \textup{Im}(\alpha))^2 \right) \notag\\
&\cdot \left(\frac{\sigma}{\pi} \int_{\mathbb{R}} \frac{\arccos\left( x/\sqrt{x^2 + \textup{Im}(\alpha)^2}  \right)^2}{(x - (\textup{Re}(\alpha) + \mu))^2 + \sigma^2} dx - \theta_{\alpha}^2 \right), 
\end{align} 
where we let  
\[ \theta_{\alpha} := \arccos\left( \frac{\mu + \textup{Re}(\alpha)}{\sqrt{(\mu + \textup{Re}(\alpha))^2 + (\sigma + \textup{Im}(\alpha))^2}}  \right). \]
\end{Thm}

The expression in \eqref{Var-Gk-Cauchy} is complicated, 
however, in the case that $\alpha = 0$, 
we see that 
\[ \lim_{n \to \infty} n \textup{Var}\left( G_n  \right) =  2r^2 \theta (\pi - \theta), \]
where we let $r \exp(i\theta) = \mu + \sigma i$  and $\theta \in (0, \pi)$.

\begin{proof}
It is easy to see \eqref{sublinear} and \eqref{radius}. 
Since
\begin{align*} 
E\left[ \left| \log(X_1 + \alpha) \right|^2\right]  &\le \pi^2 + E\left[ \left| \log(|X_1 + \textup{Re}(\alpha)|) \right|^2 \right]  \\
&= \pi^2 + E\left[ \left| \log(X_1 + \alpha) \right|^2, \ |X_1 + \textup{Re}(\alpha)| \ge 1 \right] \\
& + E\left[ \left| \log(X_1 + \alpha) \right|^2, \ |X_1 + \textup{Re}(\alpha)| \le 1 \right] \\
&\le  \pi^2 +  2 E\left[ \left|X_1 +  \textup{Re}(\alpha) \right|^{1/2}\right] + \frac{1}{\pi \sigma}  \int_{-1}^{1} (\log |x|)^2 dx < +\infty,  
\end{align*}
we have \eqref{f-L2}. 

(i) We first remark that 
if $-1 < p < 1$ and $\alpha \in \mathbb{C}$, then 
$E\left[\left| X_1 + \alpha \right|^p \right] < +\infty$.
By applying Proposition \ref{integral} to the case that 
$f(x) = x^{p}$,   
we see that 
$E\left[(X_1 + \alpha)^{p} \right] = (\mu + i \sigma + \alpha)^p$. 
Assertion (i) follows from this. 

(ii)  
The a.s. convergence part of (ii) follows from \eqref{f-SLLN}. 
Let $p \ge 1$. 
Assume that $n > p$. 
We first remark that 
\[ \left|\prod_{i=1}^{n} (X_i + \alpha)^{1/n}\right| = \prod_{i=1}^{n} |X_i + \alpha|^{1/n}. \]
Then, by the independence, 
\begin{align*} 
E\left[ \left|  \prod_{i=1}^{n} (X_i + \alpha)^{1/n} \right|^p \right] &= \prod_{i=1}^{n} E\left[|X_i + \alpha|^{p/n} \right] \\
&= \left(E\left[|X_1 + \alpha|^{p/n}\right]\right)^n \to \left|\mu+\sigma i + \alpha \right|^p, \ \ n  \to \infty.  
\end{align*}
Hence, $\left\{  \left|  \prod_{i=1}^{n} (X_i + \alpha)^{1/n} \right|^p \right\}_{n \ge 2}$ is uniformly integrable. 
By this and the a.s. convergence, 
we have the $L^p$  convergence part of (ii).  

(iii) By Theorem \ref{var-conv}, 
we see that 
\begin{equation*} 
\lim_{n \to \infty} n \textup{Var}\left( G_n^{(\alpha)}  \right) = \exp\left(2 E\left[ \log|X_1 + \alpha| \right]\right) \textup{Var}\left(\log( X_1 + \alpha)\right) 
\end{equation*} 
The rest of the proof is devoted to the calculations for $E\left[ \log|X_1 + \alpha| \right]$ and $\textup{Var}\left(\log (X_1 + \alpha) \right)$. 
By \eqref{unbiased-basic},  
we see that 
\begin{align*} 
\exp(2 E[ \log|X_1 + \alpha|]) &= \exp\left(2 \textup{Re}(E[ \log(X_1 + \alpha)]) \right)  \\ 
&=  (\mu + \textup{Re}(\alpha))^2 + (\sigma + \textup{Im}(\alpha))^2. 
\end{align*}

We will compute $\textup{Var}\left(\log(X_1 + \alpha)\right)$. 
We compute $\left| E\left[ \log(X_1 + \alpha) \right] \right|^2$. 
By \eqref{unbiased-basic},  
we obtain that 
\[ E\left[ \log\left( \sqrt{(X_1 + \textup{Re}(\alpha))^2 + \textup{Im}(\alpha)^2} \right)\right] = \log\left( \sqrt{(\mu + \textup{Re}(\alpha))^2 + (\sigma + \textup{Im}(\alpha))^2} \right) \]
and 
\[ \textup{Im}\left(E\left[\log(X_1 + \alpha)\right]\right) = \theta_{\alpha}. \]
Hence, 
\begin{equation}\label{var-polar} 
\left|E\left[\log(X_1 + \alpha)\right]\right|^2  = \log\left( \sqrt{(\mu + \textup{Re}(\alpha))^2 + (\sigma + \textup{Im}(\alpha))^2} \right)^2 +  \theta_{\alpha}^2.  
\end{equation}

We now compute $E\left[ \left|\log(X_1 + \alpha)\right|^2 \right]$. 
Let 
\[ \theta_{X_1} := \arccos\left( \frac{X_1 + \textup{Re}(\alpha)}{\sqrt{(X_1 + \textup{Re}(\alpha))^2 + \textup{Im}(\alpha)^2}}  \right). \]
Then, 
\begin{align*} 
E\left[ |\log(X_1 + \alpha)|^2 \right] 
&= E\left[ \log(|X_1 + \alpha|)^2 - \theta_X^2 \right] + 2E\left[\theta_{X_1}^2 \right] \\
&= E\left[\textup{Re}\left(  (\log(X_1 + \alpha))^2 \right) \right] + 2E\left[\theta_{X_1}^2 \right]  \\
&= \log\left( \sqrt{(\mu + \textup{Re}(\alpha))^2 + (\sigma + \textup{Im}(\alpha))^2} \right)^2 -  \theta_{\alpha}^2 +  2E\left[\theta_{X_1}^2 \right]. 
\end{align*}
By this and \eqref{var-polar}, we see that 
\begin{align*} 
\textup{Var}(\log(X_1 + \alpha)) &=  E\left[ |\log(X_1 + \alpha)|^2 \right] - |E[\log(X_1 + \alpha)]|^2 \\
&= 2\left(E\left[\theta_{X_1}^2 \right] - \theta_{\alpha}^2\right). 
\end{align*}

We see that 
\[ E\left[\theta_{X_1}^2 \right] = \frac{\sigma}{\pi} \int_{\mathbb{R}} \arccos\left( \frac{x}{\sqrt{x^2 + \textup{Im}(\alpha)^2}}  \right)^2 \frac{1}{(x - (\textup{Re}(\alpha) + \mu))^2 + \sigma^2} dx. \]
Now assertion (iii) follows. 
\end{proof}

\begin{Rem}
(i) Concerning the case that $\alpha = 0$, Zolotarev \cite[(4.1.8)]{Zolotarev1986}  showed that 
\begin{equation*}
 E[(\log |X_1|)^2]  = (\log r)^2 +\theta(\pi - \theta), \ \mu + \sigma i = r \exp(i \theta). 
\end{equation*}   
An alternative simple proof of this equality   using a contour integral in a Riemannian surface is given by \cite[Proposition 2.4]{Akaoka2021-3}. \\
(ii)  \cite[Lemma 4.1.2]{Zolotarev1986} gives a general formula for 
\[ E\left[ \left( \log |X_1| - E[\log |X_1|] \right)^l \left( \textup{sign}(X_1)  - E\left[ \textup{sign}(X_1) \right] \right)^m\right], \ l, m = 0,1, \cdots, \]
in terms of the Bell polynomials.  
\end{Rem}

\subsection{M\"obius transformation}

In this subsection, we consider the case that $f(x) = 1/(x + \alpha), \ \alpha \in \mathbb{H}$.  

\begin{Thm}[M\"obius transformation]\label{Mobius-Cauchy}
Let $\alpha \in \mathbb{H}$. 
Then, it holds that\\
(i) $C_n^{(\alpha)}$ is unbiased for $n \ge 3$.\\
(ii) The following convergence holds a.s. and 
in $L^{p}$ for every $p$: 
$$ \lim_{n \to \infty} C^{(\alpha)}_n = \mu + \sigma i.$$
(iii) 
\[ \lim_{n \to \infty} n \textup{Var}\left(C_n^{(\alpha)} \right) = \frac{\sigma}{\textup{Im}(\alpha)} \left|\mu + \sigma i + \alpha\right|^2.  \]
\end{Thm}

\begin{proof}
Since $f$ is bounded, it is easy to see  \eqref{sublinear}, \eqref{radius} and \eqref{f-L2}. 

(i)  We remark that 
\[ E\left[C_n^{(\alpha)} \right] =  -\alpha  + E\left[  \dfrac{n}{\sum_{j=1}^{n} 1/(X_j + \alpha)} \right].  \]
By the geometric-harmonic inequality, we see that 
\[ E\left[  \left| \dfrac{n}{\sum_{j=1}^{n} 1/(X_j + \alpha)} \right| \right] \le \frac{1}{\textup{Im}(\alpha)}  E\left[ \dfrac{n}{\sum_{j=1}^{n} |X_j + \alpha|^{-2}} \right] \le \frac{1}{\textup{Im}(\alpha)} E\left[ |X_1 + \alpha|^{2/n} \right]^n < +\infty. \]

We remark that for every fixed $z_2, \cdots, z_n \in \overline{\mathbb{H}}$, 
\[ \sup_{z_1 \in  \overline{\mathbb{H}}}  \left| \dfrac{n}{\sum_{j=1}^{n} 1/(z_j + \alpha)} \right| < +\infty. \]
By Fubini's theorem and  the Cauchy integral formula, 
\[ E\left[ \dfrac{n}{\sum_{j=1}^{n} 1/(X_j + \alpha)} \right] 
= \int_{\mathbb{R}^n} \dfrac{n}{\sum_{j=1}^{n} 1/(x_j + \alpha)} \prod_{j=1}^{n} \frac{\sigma}{(x_j - \mu)^2 + \sigma^2} dx_1 \cdots dx_n \]
\[ =  \mu + \sigma  i + \alpha. \]
Thus we see that $E\left[C_n^{(\alpha)} \right] = \mu + \sigma  i $. 

(ii) The a.s. convergence follows from Theorem \ref{SLLN}. 
Let $1 \le p < +\infty$. 
Then by using \eqref{Jk-integrability} for $\delta = 1/2$, 
we see that 
$\sup_{n \ge 2} E\left[ \left|J^{(\alpha)}_n \right|^{-p}\right] < +\infty$.
Hence, 
$\sup_{n \ge 2} E\left[ \left|C^{(\alpha)}_n \right|^{p}\right] < +\infty$.
Now the $L^p$-convergence part of the assertion follows from this and the a.s. convergence.

(iii) Recall the definition of $J^{(\alpha)}_n$  in \eqref{def-Jk}. 
We see that 
\[ \textup{Var}\left(C^{(\alpha)}_n \right) = \textup{Var} \left(\frac{1}{J^{(\alpha)}_n} \right) = E\left[ \frac{1}{ \left|J^{(\alpha)}_n \right|^2}\right] 
- \left| E\left[ \frac{1}{J^{(\alpha)}_n}\right] \right|^2.   \]

By (i) and the Cauchy integral formula, 
we see that for every $n$, 
\[ E\left[ \frac{1}{J^{(\alpha)}_n}\right] = \mu + \sigma  i  + \alpha  = \frac{1}{E\left[ 1/(X_1 + \alpha) \right]} =  \frac{1}{\mu^{(\alpha)}}.  \]

By  this and Proposition \ref{Ck}, we see that 
\begin{equation*}  
\lim_{n \to \infty} n  \textup{Var}\left(C^{(\alpha)}_n \right) = |\mu + \sigma  i  + \alpha|^4 \textup{Var}\left( \frac{1}{X_1 + \alpha}\right) 
\end{equation*}
and 
\begin{align*} 
\textup{Var}\left( \frac{1}{X_1 + \alpha}\right) &=  \frac{1}{\alpha - \overline{\alpha}}   E\left[\frac{1}{X_1 + \overline{\alpha}} - \frac{1}{X_1 + \alpha} \right] - \frac{1}{\left|\mu + \sigma  i  + \alpha \right|^2}  \\
&= \frac{\sigma}{ \textup{Im}(\alpha)  \left|\mu + \sigma  i  + \alpha \right|^2}.  
\end{align*}

Thus we have the assertion. 
\end{proof}

\begin{Rem}
(i) We can also show Theorem \ref{Mobius-Cauchy} (i) by using a method similar to \cite[Eq. (14)]{Pakes1999}. 
We can consider here a complex-valued version of the Laplace transform.  
It holds that for every $\beta \in \mathbb{C}$ with $\textup{Re}(\beta) > 0$, 
$\frac{1}{\beta} = \int_0^{\infty} \exp(-\beta s) ds$.
By this and Fubini's theorem, 
we see that for $n \ge 3$, 
\begin{align*} 
E\left[  \dfrac{n}{\sum_{j=1}^{n} 1/(X_j + \alpha)} \right] 
&= i E\left[  \dfrac{n}{i \sum_{j=1}^{n} 1/(X_j + \alpha)} \right]\\ 
&= i E\left[  \int_0^{\infty} \exp\left(- \frac{s i}{n} \sum_{j=1}^{n} \frac{1}{X_j + \alpha} \right) ds \right]\\
&= i \int_0^{\infty} E\left[ \exp\left(\dfrac{-si}{n(X_1 + \alpha) }\right) \right]^n ds\\
&= \mu + \sigma  i  + \alpha. 
\end{align*}
(ii) We see that 
\[ \min_{\alpha \in \mathbb{H}} \frac{\sigma}{\textup{Im}(\alpha)} |\mu + \sigma  i  + \alpha|^2 = 4\sigma^2. \]
The minimum is attained at $\alpha = -\mu + \sigma  i $ only. 
Here $4\sigma^2$ is the lower bound followed by the Cramer-Rao bound for unbiased estimators. 
See \eqref{CR}. \\
(iii) 
\[ \lim_{\textup{Im}(\alpha) \to 0}  \frac{\sigma}{\textup{Im}(\alpha)} |\mu + \sigma  i  + \alpha|^2 = +\infty. \]
(iv) The distribution of $n/ \sum_{j=1}^{n} (1/X_j)$, which is the harmonic mean of the {\it standard} Cauchy random variables, 
is identical with the standard Cauchy distribution. 
Therefore we have {\it not} considered the case that $\alpha = 0$.  \\
(v) Since $C_{n}^{(\alpha)}$ is unbiased, we do not need to use Proposition \ref{VarJ-small} in the proof of Theorem \ref{Mobius-Cauchy} above. 
\end{Rem}

{\it Acknowledgements} \ 
The authors wish to give their gratitudes to an anonymous referee for his or her comments. 
The authors wish to give their gratitudes to Prof. Ken-iti Sato for references and to Prof. Matyas Barczy for comments. 
The second and third authors were supported by JSPS KAKENHI 19K14549 and 16K05196 respectively.

\bibliographystyle{amsplain}
\bibliography{Cauchy1}

\end{document}